\documentclass[a4paper,11pt,reqno]{amsart}  
\usepackage[a4paper,hmargin={2.5cm,2.5cm},vmargin={2cm,2cm}]{geometry}

\usepackage[square]{natbib} 
\usepackage{amsmath}
\usepackage{amssymb,amsthm}
\usepackage{bbm}
\usepackage{dsfont}
\usepackage{stmaryrd}
\usepackage[all]{xy}

\usepackage{color}

\theoremstyle{plain}
\newtheorem{theorem}{Theorem}[section]
\newtheorem{lemma}[theorem]{Lemma}
\newtheorem{proposition}[theorem]{Proposition}
\newtheorem{corollary}[theorem]{Corollary}

\theoremstyle{definition}
\newtheorem{definition}[theorem]{Definition}

\theoremstyle{remark}

\newenvironment{eqcond}{\begin{enumerate}}{\end{enumerate}}

\newcommand{\Rw}{\Rightarrow}

\newcommand{\hrw}{\hookrightarrow}

\newcommand{\Lrw}{\Longrightarrow}

\newcommand{\Lrlw}{\Longleftrightarrow}

\newcommand{\ff}{\mathfrak{f}}

\newcommand{\fp}{\mathfrak{p}}

\newcommand{\fx}{\mathfrak{x}}
\newcommand{\fy}{\mathfrak{y}}

\newcommand{\calA}{\mathcal{A}}

\newcommand{\calC}{\mathcal{C}}

\newcommand{\calO}{\mathcal{O}}
\newcommand{\calS}{\mathcal{S}}

\newcommand{\fF}{\mathfrak{F}}

\DeclareMathOperator{\spec}{spec}
\DeclareMathOperator{\upc}{\uparrow\!}
\DeclareMathOperator{\downc}{\downarrow\!}
\DeclareMathOperator{\Up}{P_{_{\! \uparrow}}\!}
\DeclareMathOperator{\Spl}{Spl}

\newcommand{\mate}[1]{\,^\ulcorner\! #1^\urcorner}

\newcommand{\catfont}[1]{\mathsf{#1}}

\newcommand{\SET}{\catfont{Set}}
\newcommand{\REL}{\catfont{Rel}}
\newcommand{\DIST}{\catfont{Dist}}

\newcommand{\TOP}{\catfont{Top}}
\newcommand{\ORDCH}{\catfont{OrdCompHaus}}
\newcommand{\PRIEST}{\catfont{Priest}}
\newcommand{\COMPHAUS}{\catfont{CompHaus}}
\newcommand{\COMPHAUSREL}{\catfont{CompHausRel}}
\newcommand{\STONE}{\catfont{Stone}}
\newcommand{\STONEREL}{\catfont{StoneRel}}
\newcommand{\STCOMP}{\catfont{StLocComp}}
\newcommand{\STCOMPDIST}{\catfont{StLocCompDist}}
\newcommand{\SPEC}{\catfont{Spec}}
\newcommand{\SPECDIST}{\catfont{SpecDist}}
\newcommand{\ESA}{\catfont{Esa}}
\newcommand{\ESADIST}{\catfont{EsaDist}}
\newcommand{\GESADIST}{\catfont{GEsaDist}}

\newcommand{\DLAT}{\catfont{DLat}}
\newcommand{\BOOL}{\catfont{Boole}}
\newcommand{\CONTLAT}{\catfont{ContLat}}
\newcommand{\STCONTDLAT}{\catfont{StContDLat}}
\newcommand{\COHEYT}{\catfont{coHeyt}}

\usepackage{mathtools}

\newcommand{\modto}{\mathrel{\mathmakebox[\widthof{$\xrightarrow{\rule{1.45ex}{0ex}}$}]
{\xrightarrow{\rule{1.45ex}{0ex}}\hspace*{-2.8ex}{\circ}\hspace*{1ex}}}} 
\newcommand{\relto}{\mathrel{\mathmakebox[\widthof{$\xrightarrow{\rule{1.45ex}{0ex}}$}]
{\xrightarrow{\rule{1.45ex}{0ex}}\hspace*{-2.4ex}{\mapstochar}\hspace*{1.8ex}}}} 

\newcommand{\monadfont}[1]{\mathbbm{#1}}

\newcommand{\mV}{\monadfont{V}}
\newcommand{\mtV}{\widehat{\monadfont{V}}}
\newcommand{\mF}{\monadfont{F}}
\newcommand{\mFp}{\monadfont{F}_{\! p}}

\newcommand{\fmonad}{(F,e,m)}
\newcommand{\fpmonad}{(F_p,e,m)}
\newcommand{\vmonad}{(V,e,m)}
\newcommand{\tvmonad}{(\widehat{V},e,m)}

\newcommand\adjunct[2]{\xymatrix@=8ex{\ar@{}[r]|{\top}\ar@<1mm>@/^2mm/[r]^{{#2}} & \ar@<1mm>@/^2mm/[l]^{{#1}}}}

\newcommand{\doo}[1]{\overset{\centerdot}{#1}}

\newcommand{\op}{\mathrm{op}}
\newcommand{\co}{\mathrm{co}}

\newcommand{\df}[1]{\emph{\textbf{#1}}}

\title{Some notes on Esakia spaces}
\author{Dirk Hofmann}
\author{Pedro Nora}
\thanks{Partial financial assistance by Portuguese funds through CIDMA (Center for Research and Development in Mathematics and Applications), and the Portuguese Foundation for Science and Technology (``FCT -- Funda\c{c}\~ao para a Ci\^encia e a Tecnologia''), within the project PEst-OE/MAT/UI4106/2014, and by the project NASONI under the contract FCOMP-01-0124-FEDER-028923 is gratefully acknowledged.}
\address{Center for Research and Development in Mathematics and Applications, Department of Mathematics, University of Aveiro, 3810-193 Aveiro, Portugal}
\email{dirk@ua.pt}
\email{a28224@ua.pt}
\dedicatory{Dedicated to Manuela Sobral}
\date{\today}


\usepackage{hyperref} 

\begin{document}

\begin{abstract}
Under Stone/Priestley duality for distributive lattices, Esakia spaces correspond to Heyting algebras which leads to the well-known dual equivalence between the category of Esakia spaces and morphisms on one side and the category of Heyting algebras and Heyting morphisms on the other. Based on the technique of idempotent split completion, we give a simple proof of a more general result involving certain relations rather then functions as morphisms. We also extend the notion of Esakia space to all stably locally compact spaces and show that these spaces define the idempotent split completion of compact Hausdorff spaces. Finally, we exhibit connections with split algebras for related monads.
\end{abstract}

\maketitle

\section*{Introduction}

These notes evolve around the observation that Esakia duality for Heyting algebras arises more naturally when considering the larger category $\SPECDIST$ with objects spectral spaces and with morphisms spectral distributors. In fact, as we observed already in \citep{Hof14}, in this category Esakia spaces define the idempotent split completion of Stone spaces. Furthermore, it is well-known that $\SPECDIST$ is dually equivalent to the category $\DLAT_{\bot,\vee}$ of distributive lattices and maps preserving finite suprema and that, under this equivalence, Stone spaces correspond to Boolean algebras. This tells us that the category of Esakia spaces and spectral distributors is dually equivalent to the idempotent split completion of the category $\BOOL_{\bot,\vee}$ of Boolean algebras and maps preserving finite suprema. However, the main ingredients to identify this category as the full subcategory of $\DLAT_{\bot,\vee}$ defined by all co-Heyting algebras were already provided by \citeauthor{MT46} in \citeyear{MT46}.

In order to present this argumentation, we carefully recall in Section \ref{sect:StoneHalmos} various aspects of spectral spaces and Stone spaces which are the spaces occurring on the topological side of the famous duality theorems of Stone for distributive lattices and Boolean algebras. Special emphasis is given to the larger class of stably locally compact spaces and their relationship with ordered compact Hausdorff spaces. We also briefly present the extension of Stone's result to categories of continuous relations, an idea attributed to Halmos. These continuous relations and, more generally, spectral distributors, are best understood using the Vietoris monad which is the topic of Section \ref{sect:Vietoris}. In particular, we identify adjunctions in the Kleisli category of the lower Vietors monad on the category of stably locally compact spaces and spectral maps which is then used to describe Esakia spaces as the idempotent split completion of Stone spaces. In Section \ref{sect:Esakia} we use the facts presented in the previous section to deduce Esakia dualities using the technique of idempotent split completion. Moreover, we extend the notion of Esakia space to all stably locally compact spaces and deduce in Section \ref{sect:GenEsakia} that the category of (generalised) Esakia spaces and spectral distributors is the idempotent split completion of the category of compact Hausdorff spaces and continuous relations. Finally, the idempotent split completion of Kleisli categories is ultimately linked to the notion of split algebra for a monad, which is the topic of Section \ref{sect:SplitAlgebras}.

\section{Stone and Halmos dualities}\label{sect:StoneHalmos}

The aim of this section is to collect some well-known facts about duality theory for Boolean algebras and distributive lattices and about the topological spaces which occur as their duals. As much as possible we try to indicate original sources.

Naturally, we begin with the classical Stone dualities stating (in modern language) that the category $\STONE$ of \df{Stone spaces} (= zero-dimensional compact Hausdorff topological spaces) and continuous maps is dually equivalent to the category $\BOOL$ of Boolean algebras and homomorphisms (see \citep{Sto36})
\[
 \STONE^\op\simeq\BOOL;
\]
and that the category $\SPEC$ of spectral spaces and spectral maps is dually equivalent to the category $\DLAT$ of distributive lattices\footnote{We note that for us a lattice is an ordered set with finite suprema and finite infima, hence every lattice has a largest element $\top$ and a smallest element $\bot$.} and homomorphisms (see \citep{Sto38})
\[
 \SPEC^\op\simeq\DLAT.
\]
We recall that a topological space $X$ is \df{spectral} whenever $X$ is sober and the compact and open subsets are closed under finite intersections and form a base for the topology of $X$. Note that in particular every spectral space is compact. A continuous map $f\colon X\to Y$ between spectral spaces is called \df{spectral} whenever $f^{-1}(A)$ is compact, for every $A\subseteq Y$ compact and open. A subset of a Stone space is compact if and only if it is closed, hence every Stone space is spectral and every continuous map between Stone spaces is spectral; that is, $\STONE$ is a full subcategory of $\SPEC$. Moreover, a spectral space $X$ is a Stone space if and only if $X$ is Hausdorff. Under the equivalence above, a spectral space $X$ corresponds to the distributive lattice of compact and open subsets of $X$ ordered by inclusion; if $X$ is a Stone space, then the lattice of compact opens is actually the Boolean algebra of closed and open subsets. In the other direction, to a distributive lattice $L$ one associates its prime spectrum $\spec L$; and $\spec L$ is Hausdorff if and only if $L$ is a Boolean algebra. For a detailed presentation of these duality theorems and many of their consequences we refer to \citep{Joh86}.

Another important aspect of spectral spaces is disclosed in \citep{Hoc69}: besides being the prime spectra of distributive lattices, spectral spaces are also precisely the prime spectra of commutative rings with unit. For a common study of lattice spectra and ring spectra we refer to \citep{Sim80}. \citeauthor{Hoc69} also constructs a right adjoint $\SPEC\to\STONE$ to the inclusion functor $\STONE\hrw\SPEC$ which associates to a spectral space $X$ the topological space with the same underlying set and with the topology generated by the open subsets and the complements of the compact open subsets of $X$, this space is called the \df{patch} of $X$. Furthermore, in this paper we follow \citep{Hoc69} and consider the natural \df{underlying order} of a topological T0-space $X$ defined as
\[
 x\le y\hspace{1em}\text{whenever}\hspace{1em} y\in\overline{\{x\}},
\]
which is equivalent to saying that the principal filter $\doo{x}$ converges to $y$. This order relation is discrete if and only if $X$ is T1. 

For a general (not necessarily T0) topological space, this relation is still reflexive and transitive and leads to another important feature of the category $\TOP$ of topological spaces and continuous maps: $\TOP$ is a 2-category. Here, for continuous maps $f,g\colon X\to Y$ between topological spaces we write $f\le g$ whenever $f(x)\le g(x)$ for all $x\in X$, which defines the 2-cells in $\TOP$. Consequently, we consider also subcategories of $\TOP$ as 2-categories; and note that this structure becomes trivial in $\COMPHAUS$ and $\STONE$. The 2-categorical nature of $\TOP$ leads us to consider the notion of adjunction: for continuous maps $f\colon X\to Y$ and $g\colon Y\to X$, we say that $f$ is \df{left adjoint} to $g$, written as $f\dashv g$, if $1_X\le gf$ and $fg\le 1_Y$. Given $f$, there exists up to equivalence at most one such $g$, and in this case we call $f$ a left adjoint continuous map. 

\citeauthor{Hoc69} also introduces a notion of \df{dual space}: for a spectral space $X$, the set $X$ equipped with the topology generated by the complements of the compact open subsets of $X$ is a spectral space whose underlying order is dual to the underlying order of $X$. Therefore we denote this space by $X^\op$, and it is not hard to see that $(X^\op)^\op=X$. Since every spectral map $f\colon X\to Y$ is also a spectral map of type $X^\op\to Y^\op$, we obtain a 2-functor
\[
 (-)^\op\colon \SPEC\to\SPEC^\co.
\]

A different perspective on spectral spaces is offered by \cite{Pri70,Pri72}. \citeauthor{Pri70} showed that the category $\DLAT$ is also dually equivalent to the full subcategory $\PRIEST$ of the category $\ORDCH$ of ordered compact Hausdorff spaces (as introduced in \citep{Nac50}) and continuous monotone maps defined by all order-separated spaces (these spaces are nowadays called \df{Priestley spaces}). We note that the compact Hausdorff topology of a Priestley space is necessarily a Stone topology. Hence, in an indirect way she showed that the categories $\SPEC$ and $\PRIEST$ are equivalent,
\[
 \SPEC\simeq\PRIEST.
\]
A couple of years later, \citeauthor{Cor75} proved this fact directly in \citep{Cor75} (see also \citep{Fle00}); in fact, both categories are shown to be isomorphic. Here a spectral space $X$ corresponds to the Priestley space with the same underlying set, ordered by the underlying order of $X$, and equipped with the patch topology. In the other direction, a Priestley space $X$ corresponds to the spectral space whose topology is given by all those opens of $X$ which are also down-closed.

More generally, this construction does not only apply to Priestley spaces but indeed to all ordered compact Hausdorff spaces and defines an isomorphism between $\ORDCH$ and the category $\STCOMP$ of stably locally compact spaces and spectral maps between them. Below we sketch this correspondence, for more information we refer to \citep{GHK+80} and also to the more recent \citep{Jun04}. A topological space $X$ is called \df{stably locally compact} if $X$ is sober, locally compact and finite intersections of compact down-sets (with respect to the underlying order of $X$) are compact. A continuous map $f\colon X\to Y$ between stably locally compact spaces is \df{spectral} whenever $f^{-1}(A)$ is compact, for every $A\subseteq Y$ compact and down-closed. Equivalently, a topological space $X$ is stably locally compact if and only if $X$ is T0, locally compact and every ultrafilter in $X$ has a smallest convergence point with respect to the underlying order of $X$; and a map  $f\colon X\to Y$ between stably locally compact spaces is spectral (in particular continuous) if and only if $f$ is monotone with respect to the underlying orders and, moreover, the diagram
\[
 \xymatrix{UX\ar[r]^{Uf}\ar[d]_\alpha & UY\ar[d]^\beta\\ X\ar[r]_f & Y}
\]
commutes. Here $U\colon \SET\to\SET$ denotes the ultrafilter functor and the maps $\alpha\colon UX\to X$ and $\beta\colon UY\to Y$ pick, for each ultrafilter, the smallest convergence point. Every spectral space is stably locally compact, and the two notions of spectral maps between spectral spaces are actually equivalent. We also point out that a continuous map $f\colon X\to Y$ between stably locally compact spaces which is left adjoint in $\TOP$ is automatically spectral.

Every compact Hausdorff space is stably locally compact and every continuous map between compact Hausdorff spaces is spectral, which defines the inclusion functor 
\[
 \COMPHAUS\hrw\STCOMP.
\]
As for spectral spaces and Stone spaces, this functor has a right adjoint
\[
 \STCOMP\to\COMPHAUS
\]
which sends a stably locally compact space $X$ to the compact Hausdorff space with the same underlying set and the (generalised) patch topology, that is, the topology generated by the open subsets and the complements of the compact down-closed subsets of $X$. For $X$ spectral, this topology coincides with the patch topology described above. Using this generalised patch topology, the correspondence between spectral spaces and Priestley spaces extends immediately: every stably locally compact space $X$ defines an ordered compact Hausdorff space with the patch topology and the underlying order of $X$, and an ordered compact Hausdorff space $X$ becomes a stably locally compact space where the topology is given by all down-closed opens of $X$. Clearly, $\ORDCH$ is also a 2-category with the point-wise order on maps; and then the above described isomorphism
\[
 \ORDCH\simeq\STCOMP
\]
is an isomorphism of 2-categories. In terms of ordered compact Hausdorff spaces, the adjunction
\[
 \STCOMP\adjunct{}{}\COMPHAUS
\]
becomes
\[
 \ORDCH\adjunct{\text{discrete}}{\text{forgetful}}\COMPHAUS.
\]
Moreover, there is a 2-functor
\[
 (-)^\op\colon\ORDCH\to\ORDCH^\co
\]
which inverts the order relation of a compact Hausdorff space $X$; and which induces a 2-functor
\[
 (-)^\op\colon\STCOMP\to\STCOMP^\co
\]
where $X^\op$ turns out to be the space with the same underlying set as $X$ and with the topology induced by the complements of the compact down-sets of $X$. If $X$ is spectral, this notion of dual space coincides with the one of \citeauthor{Hoc69} described above.

Another interesting generalisation of Stone's duality theorem is given in \citep{Hal56}. Instead of continuous maps, \citeauthor{Hal56} considers continuous relations and shows that the category $\STONEREL$ of Stone spaces and continuous relations (called Boolean relations in \citep{Hal56}) is dually equivalent to the category $\BOOL_{\bot,\vee}$ of Boolean algebras with ``hemimorphisms'', that is, maps preserving finite suprema but not necessarily finite infima (see also \citep{SV88}). Similarly, the category $\SPECDIST$ of spectral spaces and spectral distributors (respectively Priestley spaces and Priestley distributors, see \citep{CLP91}) is dually equivalent to the category $\DLAT_{\bot,\vee}$ of distributive lattices and maps preserving finite suprema. 

We have not yet explained the meaning of continuous relation and spectral distributor, which is the subject of the next section.

\section{Vietoris monads}\label{sect:Vietoris}

Similarly to the fact that the category $\REL$ of sets and relations can be seen as the Kleisli category of the power-set monad on $\SET$, we will describe $\STONEREL$ and $\SPECDIST$ as Kleisli categories of certain monads. 

Before doing so, we recall the notion of ``monotone relation'' between ordered sets. A relation $r\colon X\relto Y$ between ordered sets is called a \df{distributor} whenever, for all $x,x'\in X$ and $y,y'\in Y$,
\begin{align*}
 (x\,r\,y\;\&\;y\le y')\;\Rw\;x\,r\,y' &&\text{and}&&
 (x\le x'\;\&\;x'\,r\, y')\;\Rw\;x\,r\,y'.
\end{align*}
Put differently, the corresponding map $\mate{r}\colon X\to PY$ from $X$ into the powerset of $Y$ has its image in the ordered set $\Up Y$ of all up-closed subsets of $Y$ (ordered by inverse inclusion), and the restriction $\mate{r}\colon X\to \Up Y$ is monotone. We write $r\colon X\modto Y$ to indicate that $r$ is a distributor. The relational composite of distributors is a distributor again, and the identity with respect to this composition law is the order relation on an ordered set $X$. We have thus described the category $\DIST$ of ordered sets and distributors which becomes a 2-category when considering the inclusion order of relations. We also note that $\DIST$ is isomorphic to the Kleisli category of the up-set monad on the category of ordered sets and monotone maps. A monotone map $f\colon X\to Y$ between ordered sets induces distributors $f_*\colon X\modto Y$ and $f^*\colon Y\modto X$ defined by
\begin{align*}
 x\,f_*\,y \text{ whenever } f(x)\le y &&\text{and}&& y\,f^*\,x \text{ whenever } y\le f(x)
\end{align*}
respectively; that is, $f_*=\le_Y\cdot f$ and $f^*=f^\circ\cdot\le_Y$. We also remark that $f_*\dashv f^*$ in the ordered category $\DIST$, in fact, every adjunction in $\DIST$ is of this form (see \citep{BD86}).

Arguably, the topological counterpart to the up-set monad is the \df{lower Vietoris monad} $\mV=\vmonad$ on $\TOP$ which consists of the functor $V\colon \TOP\to\TOP$ sending a topological space $X$ to the space
\[
 VX=\{A\subseteq X\mid A\text{ is closed}\}
\]
with the topology generated by the sets
\[
 U^\Diamond=\{A\in VX\mid A\cap U\neq\varnothing\}\hspace{2em}(\text{$U\subseteq X$ open}),
\]
and $Vf\colon VX\to VY$ sends $A$ to $\overline{f[A]}$, for $f\colon X\to Y$ in $\TOP$; and the unit $e$ and the multiplication $m$ of $\mV$ are given by
\begin{align*}
 e_X\colon X\to VX,\,x\mapsto\overline{\{x\}} &&\text{and}&& m_X\colon VVX\to VX,\,\calA\mapsto\bigcup\calA
\end{align*}
respectively. We note that the underlying order of $VX$ is the opposite of subset inclusion, that is, $A\le B$ if and only if $A\supseteq B$, for all $A,B\in VX$.
We also note that $V\colon \TOP\to\TOP$ is a 2-functor. The following lemma describes the convergence in $VX$ (see \citep{Hof14}).

\begin{lemma}\label{lem:VietorisConv}
Let $X$ be a topological space, $A\in VX$ and $\fp$ be an ultrafilter on $VX$. Then $\fp\to A$ in $VX$ if and only if $A\subseteq \bigcap_{\calA\in\fp}\overline{\bigcup\calA}$.
\end{lemma}

A continuous map $f\colon X\to Y$ between topological spaces is called \df{down-wards open} whenever, for every open subset $A\subseteq X$, the down-closure $\downc f[A]$ of $f[A]$ is open in $Y$. Below we record some important properties of $\mV$, for more information we refer to \cite{Sch93}, \citep{Esc98} and \cite{Hof14}.

\begin{proposition}
\phantomsection\label{prop:PropLowVietoris}
\begin{enumerate}
\item The monad $\mV=\vmonad$ on $\TOP$ is of Kock-Z\"oberlein type, that is, $e_{VX}\le Ve_X$ for every topological spaces $X$ (see \cite{Koc95} and \cite{Zob76}).
\item Let $f\colon X\to Y$ be in $\TOP$. Then $Vf$ has a left adjoint if and only if $f$ is down-wards open.
\item For every topological space $X$, if $X$ is stably locally compact, then so is $VX$.
\item If $X$ is stably locally compact, then $e_X\colon X\to VX$ and $m_X\colon VVX\to VX$ are spectral.
\item If $f\colon X\to Y$ is a continuous map between stably locally compact spaces, then $Vf\colon VX\to VY$ is spectral if and only if $f\colon X\to Y$ is spectral.
\item A stably compact space $X$ is spectral if and only if $VX$ is spectral.
\end{enumerate}
\end{proposition}

Consequently, the monad $\mV=\vmonad$ on $\TOP$ restricts to Kock-Z\"oberlein monads on $\STCOMP$ and on $\SPEC$, also denoted by $\mV=\vmonad$. Using the adjunction
\begin{align*}
 \STCOMP\adjunct{}{}\COMPHAUS,
\end{align*}
we can transfer the monad $\mV$ on $\STCOMP$ to the \df{Vietoris monad} $\mtV=\tvmonad$ on $\COMPHAUS$. Hence, $\widehat{V}X$ is the patch space of $VX$; the topology of $\widehat{V}X$ is generated by the sets
\begin{align*}
 U^\Diamond\hspace{1em}\text{($U\subseteq X$ open)} &&\text{and}&& \{A\subseteq X\text{ closed}\mid A\cap K=\varnothing\}\hspace{1em}\text{($K\subseteq X$ compact)}.
\end{align*}
We note that this is the topology on the set of closed subsets of a compact Hausdorff space originally considered by \citep{Vie22}. The unit $e$ and the multiplication $m$ are as above, but note that $e_X(x)=\{x\}$ since $X$ is Hausdorff. 

\begin{proposition}
A compact Hausdorff space $X$ is a Stone space if and only if $\widehat{V}X$ is a Stone space.
\end{proposition}

Therefore the monad $\mtV$ on $\COMPHAUS$ restricts to a monad on $\STONE$ which we also denote by $\mtV=\tvmonad$.

For a compact Hausdorff space $X$, a map $r\colon X\to\widehat{V}X$ is continuous if and only if $r\colon X\to VX$ is spectral, hence $\COMPHAUS_{\mtV}$ can be considered as a full subcategory of $\STCOMP_\mV$ and consequently $\STONE_{\mtV}$ as a full subcategory of $\SPEC_\mV$. A relation $r\colon X\relto Y$ between spectral spaces is called \df{spectral distributor}, indicated as $r\colon X\modto Y$, whenever $r$ corresponds to a morphism in $\STCOMP_\mV$, that is, the map $\mate{r}\colon X\to PY$ factors as $X\to VY\hrw PY$ and, moreover, $X\to VY$ is spectral. Then $r\colon X\modto Y$ is also a distributor between the underlying ordered sets of $X$ and $Y$, which justifies our nomenclature. Furthermore, relational composition of spectral distributors corresponds to composition in $\STCOMP_\mV$ since the lower Vietoris functor on $\STCOMP$ ``behaves like the up-set functor'', that is, for a spectral map $f\colon X\to Y$ and $A\subseteq X$, one has $\overline{f[A]}=\upc f[A]$. Therefore $\STCOMP_\mV$ is isomorphic to the category $\STCOMPDIST$ of stably locally compact spaces and spectral distributors, with relational composition and the identity on $X$ given by the underlying order relation of $X$. The category $\STCOMPDIST$ becomes a 2-category via the inclusion order of relations which is dual to the order in $\STCOMP_\mV$, that is,
\[
 \STCOMP_\mV\simeq \STCOMPDIST^\co;
\]
and we have a forgetful 2-functor
\[
 \STCOMPDIST\to\DIST.
\]
For a compact Hausdorff space $X$, the underlying order is discrete and therefore we write $\COMPHAUSREL$ to denote the full subcategory of $\STCOMPDIST$ defined by compact Hausdorff spaces. We refer to the morphisms in $\COMPHAUSREL$ as \df{continuous relations}, and write $r\colon X\relto Y$ in this case. Clearly, there is a canonical forgetful functor $\COMPHAUSREL\to\REL$. Finally, we denote by $\SPECDIST$ the full subcategory of $\STCOMPDIST$ defined by all spectral spaces; likewise, $\STONEREL$ denotes the full subcategory of $\COMPHAUSREL$ defined by all Stone spaces.

Below we give a characterisation of spectral distributors in terms of ultrafilter convergence (see \citep{Hof14}). Before doing so, we recall from \citep{Bar70} that the ultrafilter functor $U\colon\SET\to\SET$ extends to a functor $U\colon \REL\to\REL$; here, for a relation $r\colon X\relto Y$, the relation $Ur\colon UX\relto UY$ is given by
\[
 \fx\,Ur\,\fy \iff \forall A\in\fx\,.\,\{y\in Y\mid x\,r\,y\text{ for some $x\in A$}\}\in\fy
\]
for all $\fx\in UX$ and $\fy\in UY$ (see also \citep{CH04}).

\begin{proposition}
Let $X$ and $Y$ be stably locally compact spaces with ultrafilter convergence $a\colon UX\relto X$ and $b\colon UY\relto Y$ respectively. Then a relation $r\colon X\relto Y$ is a spectral distributor $r\colon X\modto Y$ if and only if $r$ is a distributor between the underlying ordered sets and the diagram of relations
\[
 \xymatrix{UX\ar|-{\object@{|}}[d]_a\ar|-{\object@{|}}[r]^{Ur} & UY\ar|-{\object@{|}}[d]^b\\
 X\ar|-{\object@{|}}[r]_r & Y}
\]
commutes.
\end{proposition}

For a spectral map $f\colon X\to Y$ between stably locally compact spaces, the spectral distributor corresponding to the composite $X\xrightarrow{f}Y\xrightarrow{e_Y}VY$ is given by $f_*\colon X\modto Y$, defined with respect to the underlying orders. Also note that the definition of $f_*$ can be applied to any map $f\colon X\to Y$, not only to monotone and spectral maps. However, we have: 

\begin{proposition}
Let $X$ and $Y$ be stably locally compact spaces and $f\colon X\to Y$ be a map. Then $f$ is spectral if and only if $f_*$ is a spectral distributor.
\end{proposition}
\begin{proof}
Clearly, if $f$ is spectral, then $f_*$ is a spectral distributor. Assume now that $f_*$ is a spectral distributor. Then $f$ is certainly monotone. The convergence $a\colon UX\relto X$ of $X$ can be written as $a=\alpha_*$, where $\alpha\colon UX\to X$ is the monotone map which sends an ultrafilter $\fx\in UX$ to its smallest convergence point; similarly, $b=\beta_*$ with $\beta\colon UY\to Y$ being the monotone map sending an ultrafilter $\fy\in UY$ to its smallest convergence point. Applying $U\colon \REL\to\REL$ to the order relation on $X$ and $Y$ gives a reflexive and transitive (but not necessarily anti-symmetric) relation on $UX$ and $UY$ respectively, and then $Uf\colon UX\to UY$ is a monotone map. From $U(\le_Y\cdot f)=U(\le_Y)\cdot Uf$ it follows that $(Uf)_*=U(f_*)$, and from $f_*\cdot\alpha_*=\beta_*\cdot(Uf)_*$ we deduce that the diagram
\[
 \xymatrix{UX\ar[r]^{Uf}\ar[d]_\alpha & UY\ar[d]^\beta\\ X\ar[r]_f & Y}
\]
commutes.
\end{proof}

Similarly, for arbitrary topological spaces $X$ and $Y$ we characterise those relations $r\colon X\relto Y$ which correspond to continuous maps of type $\mate{r}\colon X\to VY$, we call such relations continuous distributors.

\begin{proposition}
Let $X$ and $Y$ be topological spaces with ultrafilter convergence $a\colon UX\relto X$ and $b\colon UY\relto Y$ respectively. Then a relation $r\colon X\relto Y$ is a continuous distributor $r\colon X\modto Y$ if and only if $r$ is a distributor between the underlying ordered sets and, moreover, 
\begin{align*}
\xymatrix{UX\ar|-{\object@{|}}[d]_a\ar|-{\object@{|}}[r]^{Ur} & UY\ar|-{\object@{|}}[d]^b\\
 X\ar|-{\object@{|}}[r]_r\ar@{}|-{\subseteq}[ur] & Y}
&&\text{and}&&
\xymatrix@C=8ex{X\ar|-{\object@{|}}[d]_\le\ar|-{\object@{|}}[r]^{(Ur)\cdot e_X} & UY\ar|-{\object@{|}}[d]^b\\
 X\ar|-{\object@{|}}[r]_r\ar@{}|-{\supseteq}[ur] & Y.}
\end{align*}
\end{proposition}
\begin{proof}
Clearly, if $r\colon X\relto Y$ is a continuous distributor, then $r$ is also a distributor between the underlying orders. Let now $x\in X$ and assume that $\fy\to y$ in $Y$ and $r(x)\in\fy$. Since $r(x)$ is closed, $y\in r(x)$ and therefore $x\le x\,r\,y$. Let now $\fx\in UX$, $x\in X$ and $y\in Y$ with $\fx\to x$ and $x\,r\,y$. Then $U\mate{r}(\fx)\to\mate{r}(x)$ and therefore, by Lemma \ref{lem:VietorisConv},
\[
 y\in\overline{\{y'\in Y\mid x'\,r\,y'\text{ for some $x'\in A$}\}}
\]
for all $A\in\fx$. Consequently, there is some $\fy\in UY$ with $\fy\to y$ and \[\{y'\in Y\mid x'\,r\,y'\text{ for some $x'\in A$}\}\in\fy\] for all $A\in\fx$, hence $\fx\,(Ur)\,\fy$. To see the reverse implication, we show first that $r(x)$ is closed, for all $x\in X$. In fact, if there is some $\fy\in UY$ with $r(x)\in\fy$ and $\fy\to y$ in $Y$, then there is some $x'\in X$ with $x\le x'\,r\,y$ and, since $r$ is a distributor, $x\,r\,y$. To see that $\mate{r}\colon X\to VY$ is continuous, assume that $\fx\to x$ in $X$. Then, for every $y\in\mate{r}(x)$, there is some $\fy\in UY$ with $\fy\to y$ and $\fx\,(Ur)\,\fy$, hence
\[
 y\in\overline{\{y'\in Y\mid x'\,r\,y'\text{ for some $x'\in A$}\}}
\]
for all $A\in\fx$. This proves $U\mate{r}(\fx)\to\mate{r}(x)$.
\end{proof}

The (order-theoretic) distributor $f^*\colon Y\modto X$ is not always a spectral distributor. In fact, in Proposition \ref{prop:PropLowVietoris} we have already characterised those spectral maps $f\colon X\to Y$ where $Vf$ has a left adjoint. Since $\mV$ is of Kock-Z\"oberlein type, it is easy to see that such a left adjoint is necessarily an algebra homomorphism, hence:

\begin{theorem}\label{thm:LeftAdjSpecRel}
For a morphism $f\colon X\to Y$ in $\STCOMP$, the following assertions are equivalent.
\begin{eqcond}
\item $f$ is down-wards open.
\item The spectral distributor $f_*\colon X\relto Y$ has a right adjoint in $\STCOMPDIST$.
\item the distributor $f^*\colon Y\modto X$ is a spectral distributor.
\end{eqcond}
\end{theorem}

\section{Esakia dualities}\label{sect:Esakia}

Besides Boolean algebras, another important class of distributive lattices is the class of Heyting algebras, which correspond under Stone (resp.\ Priestley) duality to certain spectral (resp.\ Priestley) spaces. The precise description of this correspondence dates back to 1974; in fact, quoting \citep{DG03}: ``The description of the restricted Priestley duality for Heyting algebras was first worked out by M.~Adams. The paper in which the description appeared was distributed to a number of those working on applications of Priestley duality but was never published \dots\ Esakia gives a duality for Boolean algebras with an additional closure operation and then indicates how to use this duality to obtain a duality for Heyting algebras. It should be noted that Esakia's result was obtained without reference to either distributive lattices or Priestley duality. Indeed, because the proof of his duality for Heyting algebras is indirect and missing many details, it is not immediately clear from the paper that Esakia's duality actually is the restricted Priestley duality.''. The result of Esakia mentioned above is published in \citep{Esa74}. The Priestley spaces corresponding to Heyting algebras are often called Esakia spaces (the designation Heyting spaces is used in \citep{DG03}), they are precisely those Priestley spaces $X$ where the down-closure of every open subset of $X$ is again open. Viewing $X$ as a spectral space, $X$ is an Esakia space precisely when, for every open subset $A$ of the patch space $X_p$ of $X$, its down-closure $\downc A$ is open in $X_p$; and $\downc A$ is open in $X_p$ if and only if $\downc A$ is open in $X$. 

In this section we wish to make the point that this duality for Heyting algebras arises more naturally when considering the larger category of spectral spaces and spectral distributors. For technical reasons we will consider here co-Heyting algebras, that is, distributive lattices $L$ where $L^\op$ is a Heyting algebra.

To start, we extend the notion of Esakia space to stably locally compact spaces on the obvious way.

\begin{definition}
A stably locally compact space $X$ is called an \df{Esakia space} whenever, for every open subset $A$ of the patch space $X_p$ of $X$, its down-closure $\downc A$ is open in $X$.
\end{definition}

Bearing in mind Theorem \ref{thm:LeftAdjSpecRel}, one obtains the following characterisation (see also \citep{Hof14}).

\begin{theorem}\label{thm:CharEsakia}
For a stably locally compact space $X$, the following assertions are equivalent.
\begin{eqcond}
\item $X$ is an Esakia space.
\item The spectral map $i\colon X_p\to X,\,x\mapsto x$ is down-wards open.
\item The spectral distributor $i_*\colon X_p\modto X$ has a right adjoint (necessarily given by $i^*$).
\item $X$ is a split subobject of a compact Hausdorff space $Y$ in $\STCOMPDIST$.
\end{eqcond}
If $X$ is spectral, then the space $Y$ in the last assertion can be chosen as a Stone space.
\end{theorem}

We write $\GESADIST$ to denote the full subcategory of $\STCOMPDIST$ defined by all Esakia spaces, and $\ESADIST$ stands for the full subcategory of $\GESADIST$ defined by all spectral spaces. Recall that $\SPECDIST\simeq\DLAT_{\bot,\vee}^\op$, and one easily sees that the category $\DLAT_{\bot,\vee}$ is idempotent split complete. All told:

\begin{corollary}
The category $\ESADIST$ is the idempotent split completion of $\STONEREL$.
\end{corollary}

The algebraic analogue to Theorem \ref{thm:CharEsakia} is essentially proven in \citep{MT46}. For a distributive lattice $L$, we consider its Booleanisation $j\colon L\hrw B$ which is given by any epimorphic embedding in $\DLAT$ of $L$ into a Boolean algebra $B$ (it is a completion in the sense of \citep{BGH92}). Translated to $\SPEC$, the homomorphism $j$ corresponds to the spectral map $i\colon X_p\to X$. Furthermore, every lattice homomorphism $f\colon L_1\to L_2$ extends to a homomorphism $\overline{f}\colon B_1\to B_2$ between the corresponding Boolean algebras.

\begin{theorem}
For a distributive lattice $L$, the following assertions are equivalent.
\begin{enumerate}
\item\label{Heyt:cond1} $L$ is a co-Heyting algebra.
\item\label{Heyt:cond2} The lattice homomorphism $j\colon L\to B$ has a left adjoint in $\DLAT_{\bot,\vee}$ $j^+\colon B\to L$.
\item\label{Heyt:cond3} $L$ is a split subobject of a Boolean algebra in $\DLAT_{\bot,\vee}$.
\end{enumerate}
If $f\colon L_1\to L_2$ be a lattice homomorphism between Heyting algebras, then $f$ preserves the co-Heyting operation if and only if the diagram
\[
 \xymatrix{B_1\ar[r]^{\overline{f}}\ar[d]_{j_1^+} & B_2\ar[d]^{j_2^+}\\ L_1\ar[r]_f & L_2}
\]
commutes.
\end{theorem}
\begin{proof}
The equivalence \eqref{Heyt:cond1}$\Lrlw$\eqref{Heyt:cond2} is shown in \citep{MT46}, and \eqref{Heyt:cond2}$\Rw$\eqref{Heyt:cond3} is obvious. To see \eqref{Heyt:cond3}$\Rw$\eqref{Heyt:cond1}, let $s\colon L\to B$ and $r\colon B\to L$ in $\DLAT_{\bot,\vee}$ with $r s=1_L$ and $B$ a co-Heyting algebra with co-Heyting operation $a\rightarrowtail b$, for $a,b\in B$. For $x,y\in H$, put
\[
x\rightarrowtriangle y=r(s(x)\rightarrowtail s(y)).
\]
Then, for all $x,y,z\in L$,
\begin{align*}
y\le x\vee z &\;\Lrlw\; s(y)\le s(x)\vee s(z)\\
 &\;\Lrlw\; s(x)\rightarrowtail s(y)\le s(z)\\
 &\;\,\Lrw\; x\rightarrowtriangle y=r(s(x)\rightarrowtail s(y))\le rs(z)=z.
\end{align*}
To conclude the missing implication, just observe that $s(x)\rightarrowtail s(y)\le sr(s(x)\rightarrowtail s(y))$. The second statement is clear.
\end{proof}

We denote the full subcategory of $\DLAT_{\bot,\vee}$ defined by all co-Heyting algebras by $\COHEYT_{\bot,\vee}$. 

\begin{corollary}
The category $\COHEYT_{\bot,\vee}$ is the idempotent split completion of $\BOOL_{\bot,\vee}$.
\end{corollary}

In the sequel $\ESA$ denotes the full subcategory of $\SPEC$ defined by Esakia spaces, and $\COHEYT$ the full subcategory of $\DLAT$ defined by co-Heyting algebras. From the discussion above we obtain the main result of this section.
\begin{theorem}
The equivalence $\SPECDIST\simeq\DLAT_{\bot,\vee}^\op$ restricts to an equivalence \[\ESADIST\simeq\COHEYT_{\bot,\vee}^\op\] which, when restricted to maps and lattice homomorphisms, yields $\ESA\simeq\COHEYT^\op$. Moreover, a morphism $f\colon L_2\to L_1$ in $\COHEYT$ preserves the co-Heyting operation if and only if the corresponding spectral map $g\colon X_1\to X_2$ makes the diagram of spectral distributors
\[
 \xymatrix{X_1\ar|-{\object@{o}}[d]_{i_1^*}\ar|-{\object@{o}}[r]^{g_*} & X_2\ar|-{\object@{o}}[d]^{i_2^*}\\
	    (X_1)_p\ar|-{\object@{o}}[r]_{g} & (X_2)_p}
\]
commutative; element-wise: for all $x\in X_1$ and $y\in X_2$ with $g(x)\le y$, there is some $x'\in X_1$ with $x\le x'$ and $g(x')=y$.
\end{theorem}

\section{Generalised Esakia spaces as idempotent split completion}\label{sect:GenEsakia}

With the results of the last section in mind, we would like to conclude that $\GESADIST$ is the idempotent split completion of the category $\COMPHAUSREL$. This follows indeed from Theorem \ref{thm:CharEsakia}, as soon as we know that the category $\STCOMPDIST$ is idempotent split complete. Similarly to the case of spectral spaces, it is easier to argue in the dual category; and the following result is essentially in \citep{JKM01}. We write $\STCONTDLAT_{\bigvee,\ll}$ to denote the category of continuous distributive lattices where the way-below relation is stable under finite infima and maps preserving suprema and the way-below relation. Note that every continuous distributive lattice is a frame.

\begin{theorem}
The category $\STCOMPDIST$ is dually equivalent to the category $\STCONTDLAT_{\bigvee,\ll}$.
\end{theorem}

\begin{proposition}
The category $\STCONTDLAT_{\bigvee,\ll}$ is idempotent split complete.
\end{proposition}
\begin{proof}
Let $e\colon L\to L$ be an idempotent morphism in $\STCONTDLAT_{\bigvee,\ll}$. Then $e$ splits in the category of sup-lattices and sup-preserving maps, that is, there is a complete lattice $M$ and sup-preserving maps $r\colon L\to M$ and $s\colon M\to L$ so that $e=s r$ and $r s=1_M$. Then $M$ is certainly a distributive lattice, and, since the embedding $s\colon M\to L$ preserves suprema, for all $x,y\in M$ one has
\[
 s(x)\ll s(y)\hspace{1em}\Rw\hspace{1em} x\ll y.
\]
Consequently, since $e\colon L\to L$ preserves the way-below relation, so does $r\colon L\to M$. We show now that $s\colon M\to L$ preserves the way-below relation. To this end, let $x\ll y$ in $M$. Since $L$ is a continuous lattice,
\[
 s(y)=\bigvee\{b\in L\mid b\ll s(y)\},
\]
and note that $\{b\in L\mid b\ll s(y)\}$ is directed. Hence, $y=r s(y)$ is the directed supremum of $\{r(b)\in L\mid b\ll s(y)\}$. Therefore there exist some $b\ll s(y)$ with $x\le r(b)$ and, since $e$ preserves the way-below relation, we obtain
\[
 s(x)\le sr(b)=e(b)\ll e(s(y))=s(y).
\]
This shows that $s$ preserves the way-below relation, and from that it follows that $M$ is a continuous lattice. Finally, we prove that the way-below relation in $M$ is stable under finite infima. Note that $r(\top)=\top$ since $r$ is surjective, therefore, since $\top\ll\top$ in $L$, we obtain $\top\ll\top$ in $M$. Let now $x\ll x'$ and $y\ll y'$ in $M$. Then
\[
 x\wedge y=r(s(x)\wedge s(y))\ll r(s(x')\wedge s(y'))=x'\wedge y'.\qedhere
\]
\end{proof}

\begin{corollary}
The category $\GESADIST$ is the idempotent split completion of $\COMPHAUSREL$.
\end{corollary}

\section{Split algebras}\label{sect:SplitAlgebras}

In the previous section we have seen that the Kleisli category $\STCOMP_\mV$ is idempotent split complete; consequently, the full subcategory of $\STCOMP^\mV$ defined by the free algebras $VY$ is idempotent split complete. This seems to be a rare case, and in general the idempotent split completion of the free algebras for a monad defines an interesting class of algebras, called split algebras (see \citep{RW04}). Most notably, for the up-set monad\footnote{equivalently, the down-set monad} on the category of ordered sets and monotone functions, the split algebras are precisely the completely distributive complete lattices (see \cite{FW90} and \citep{RW94}). In this section we will relate the free algebras for the lower Vietoris monad on $\STCOMP$ with the split algebras for other monads in topology.

To start, recall that the \df{filter monad} $\mF=\fmonad$ on $\TOP$ consists of 
\begin{itemize}
\item the functor $F\colon \TOP\to\TOP$ where, for a topological space $X$, $FX$ is the set of all filters on the lattices of opens of $X$ equipped with the topology generated by the sets $A^\#=\{\ff\in FX\mid A\in\ff\}$ ($A\subseteq X$ open), and, for $f\colon X\to Y$, the map $Ff$ sends a filter $\ff$ on the opens of $X$ to the filter $\{B\subseteq Y\mid f^{-1}[B]\in\ff\}$ on the opens of $Y$;
\item and the natural transformations $e\colon 1\to F$ and $m\colon FF\to F$ are given by 
\begin{align*}
&&e_X(x)&=\doo{x}=\{A\subseteq X\mid x\in A\}&\text{and}&&m_X(\fF)&=\{A\subseteq X\mid A^\#\in\fF\}, 
\end{align*}
for all topological spaces $X$, $\fF\in FFX$ and $x\in X$.
\end{itemize}
We also note that the filter monad on $\TOP$ is of Kock-Z\"oberlein type, dual to the case of the lower Vietoris monad: $Fe_X\le e_{FX}$ for all topological spaces $X$. A filter $\ff$ in the lattice of opens of $X$ is called \df{prime} whenever, for $A,B\subseteq X$ open, $A\cup B\in \ff$ implies $A\in\ff$ or $B\in\ff$. We denote by $\mFp=\fpmonad$ the submonad of $\mF$ of defined by prime filters. For more information we refer to \citep{Esc97}.

As shown in \citep{Day75}, the category of Eilenberg--Moore algebras of the filter monad $\mF$ on $\TOP$ is equivalent to the category $\CONTLAT$ of continuous lattices and Scott-continuous and $\inf$-preserving maps; this latter category is introduced in \citep{Sco72}. We also recall from \citep{Sim82} that $\TOP^{\mFp}$ is the category $\STCOMP$; in particular, there is a forgetful functor $\TOP^\mF\to\STCOMP$. In \citep{Wyl81} it is shown that the category of Eilenberg--Moore algebras for the Vietoris monad $\mtV$ on the category $\COMPHAUS$ of compact Hausdorff spaces and continuous maps is equivalent to $\CONTLAT$, hence:
\[
 \TOP^\mF\simeq\COMPHAUS^{\mtV}.
\]
A slight generalisation of this result is presented in \citep{Hof14}.

\begin{proposition}
The functor $(-)^\op\colon\STCOMP\to\STCOMP$ restricts to an equivalence $\TOP^\mF\simeq\STCOMP^\mV$.
\end{proposition}

Since $\STCOMP\simeq\TOP^{\mFp}$, we conclude that the filter monad $\mF$ on $\TOP$ is isomorphic to the composite monad of $\mFp$ and $\mV$ via a distributive law, and that $FX=(V(F_p(X)^\op))^\op$ for each topological space $X$.

Following \citep{RW04}, for a monad $\monadfont{D}$ on a category $\catfont{C}$ where idempotents split we consider the full subcategory $\Spl(\catfont{C}^\monadfont{D})$ of $\catfont{C}^\monadfont{D}$ defined by the \df{split algebras}, that is, by those $\monadfont{D}$-algebras $X$ with algebra structure $\alpha\colon DX\to X$ for which exists a homomorphism $t\colon X\to DX$ with $\alpha t=1_X$. If $\monadfont{D}$ is of Kock-Z\"oberlein type, these splittings are adjoint to the algebra structure. Put differently, $\Spl(\catfont{C}^\monadfont{D})$ is the idempotent split completion of the full subcategory of $\catfont{C}^\monadfont{D}$ defined by the free algebras.

The split algebras for various filter monads are studied in \citep{Hof13a} where they are characterised by a disconnectedness condition. In particular:

\begin{proposition}
The split algebras for the prime filter monad are precisely those stably compact spaces $X$ where, for every open subset of $X$, its closure in the patch topology is open in $X$. The split algebras for the filter monad are precisely the filter spaces of frames.
\end{proposition}

We note that every split algebra for $\mFp$ is a spectral space. The following lemma is easy to prove (see \citep{Hof13a}, for instance).

\begin{lemma}\label{lem:FrameInDLat}
Let $L$ be a distributive lattice. Then the following assertions are equivalent.
\begin{eqcond}
\item $L$ is a coframe.
\item $L$ is a split subobject in $\DLAT$ of the lattice $\calC(X)$ of closed subsets, for some topological space $X$.
\item $L$ is a split subobject in $\DLAT_{\bot,\vee}$ of the lattice $\calC(Y)$ of closed subsets, for some topological space $Y$.
\end{eqcond}
\end{lemma}

The result above allows us to recover the description of frames inside Priestley (resp.\ Stone) duality of \citep{PS88}. Note that those spaces which are split algebras for $\mFp$ are called f-spaces there.

\begin{corollary}
A stably compact space $X$ is a split algebra for $\mFp$ if and only if $X$ is spectral and its lattice $\calS X$ of compact open subsets of $X$ is a frame.
\end{corollary}
\begin{proof}
First note that every split algebra $X$ for $\mFp$ is spectral since $X$ is a subspace of a free algebra $F_pY$ ($Y$ in $\TOP$) in $\STCOMP\simeq\TOP^{\mFp}$. Since $\SPEC\simeq\DLAT^\op$, a spectral space $X$ is a split subobject of some $F_pY$ in $\SPEC$ if and only if $\calS X$ is a split subobject of $\calS F_p(Y)\simeq\calO Y$ in $\DLAT$.
\end{proof}

On the other side, the split algebras for the lower Vietoris monad on $\STCOMP$ are precisely the free algebras $VY$ since $\STCOMP_\mV$ is idempotent split complete. It is also observed in \citep{Hof14} that the equivalence functor $(-)^\op\colon \TOP^\mF\to \STCOMP^\mV$ restricts to a functor
\[
 (-)^\op\colon\Spl(\TOP^\mF)\to\Spl(\STCOMP^\mV);
\]
but this functor is not an equivalence since $X^\op$ is spectral for every $X$ in $\Spl(\TOP^\mF)$, but $VY$ is spectral if and only if $Y$ is spectral (see Proposition \ref{prop:PropLowVietoris}). This leaves us with the problem to characterise those spectral spaces $X$ so that $VX$ is in the image of the functor above.

\begin{theorem}
Let $X$ be a spectral space. Then $(VX)^\op$ is a split algebra for the filter monad $\mF$ on $\TOP$ if and only if $X^\op$ is a split algebra for the prime filter monad $\mFp$ on $\TOP$.
\end{theorem}
\begin{proof}
For a spectral space $X$, $(VX)^\op$ is a split algebra for $\mF$ if and only if there is a topological space $Y$ and morphisms 
\[
 \xymatrix@=8ex{(VX)^\op\ar@<1mm>[r]^-{s} & \ar@<1mm>[l]^-{r} (V((F_p Y)^\op))^\op}
\]
in $\TOP^\mF$ with $r s =1$. Hence,
\[
 \xymatrix@=8ex{VX\ar@<1mm>[r]^-{s^\op} & \ar@<1mm>[l]^-{r^\op} V((F_p Y)^\op)}
\]
$\SPEC^\mV$, which corresponds to spectral distributors
\[
 \xymatrix@=8ex{X\ar@<1mm>|-{\object@{o}}[r]^-{\varphi} & \ar|-{\object@{o}}@<1mm>[l]^-{\psi} (F_p Y)^\op}
\]
with $\psi\cdot\varphi=1_X$ in $\SPECDIST$. Since $\SPECDIST\simeq\DLAT_{\bot,\vee}^\op$, this is equivalent to $\calS X$ being a split subobject of $(\calS F_p Y)^\op\simeq\calC Y$ in $\DLAT_{\bot,\vee}^\op$, which in turn is equivalent to $\calS(X^\op)$ being a frame.
\end{proof}

%

\end{document}